\documentclass[12pt,reqno,draft]{amsart} 
\usepackage{amssymb,amscd,url}

\begin{document}

%%%%%%%%%%%%%%%%%%%%%%%%%%%%%%%%%%%%%%%%%%%%%%%%%%%%%%%%%%%%%%%%%%%%%%
%% Title and Author Information

\title[Primitive Divisors and Dynamical Zsigmondy Sets]
{Primitive Divisors, Dynamical Zsigmondy Sets, and Vojta's Conjecture}
%: Started August 20, 2012
\date{\today}
\author[Joseph Silverman]{Joseph H. Silverman}
\email{jhs@math.brown.edu}
\address{Mathematics Department, Box 1917
         Brown University, Providence, RI 02912 USA}
\thanks{The author's research partially supported by NSF DMS-0854755
and Simon's Collaboration Grant \#241309.}

%%%%%%%%%%%%%%%%%%%%%%%%%%%%%%%%%%%%%%%%%%%%%%%%%%%%%%%%%%%%%%%%%%%%%%

\allowdisplaybreaks

\hyphenation{ca-non-i-cal semi-abel-ian}

%%%%%%%%%%%%%%%%%%%%%%%%%%%%%%%%%%%%%%%%%%%%%%%%%%%%%%%%%%%%%%%%%%%%%%
% Theorem environments

\newtheorem{theorem}{Theorem}
\newtheorem{lemma}[theorem]{Lemma}
\newtheorem{conjecture}[theorem]{Conjecture}
\newtheorem{proposition}[theorem]{Proposition}
\newtheorem{corollary}[theorem]{Corollary}
\newtheorem*{claim}{Claim}

\theoremstyle{definition}
% The * surpresses numbering
\newtheorem*{definition}{Definition}
\newtheorem{example}[theorem]{Example}

\theoremstyle{remark}
\newtheorem{remark}[theorem]{Remark}
\newtheorem{question}[theorem]{Question}
\newtheorem*{acknowledgement}{Acknowledgements}

%%%%%%%%%%%%%%%%%%%%%%%%%%%%%%%%%%%%%%%%%%%%%%%%%%%%%%%%%%%%%%%%%%%%%%

%%%%%%%% Set Up Environment for Notation %%%%%%%%%%%%%%
% This is currently set to allow quite wide items to be defined
\newenvironment{notation}[0]{%
  \begin{list}%
    {}%
    {\setlength{\itemindent}{0pt}
     \setlength{\labelwidth}{4\parindent}
     \setlength{\labelsep}{\parindent}
     \setlength{\leftmargin}{5\parindent}
     \setlength{\itemsep}{0pt}
     }%
   }%
  {\end{list}}

%%%%%%%% Set Up Environment for Parts in Theorems %%%%%%%%%%%%%%
\newenvironment{parts}[0]{%
  \begin{list}{}%
    {\setlength{\itemindent}{0pt}
     \setlength{\labelwidth}{1.5\parindent}
     \setlength{\labelsep}{.5\parindent}
     \setlength{\leftmargin}{2\parindent}
     \setlength{\itemsep}{0pt}
     }%
   }%
  {\end{list}}
% Use \Part{(a)}, instead of \item[(a)], to ensure upright font
\newcommand{\Part}[1]{\item[\upshape#1]}

%%%%%%%%%%%%%%%%%%
% Greek Alphabet %
%%%%%%%%%%%%%%%%%%
\renewcommand{\a}{\alpha}
\renewcommand{\b}{\beta}
\newcommand{\g}{\gamma}
\renewcommand{\d}{\delta}
\newcommand{\e}{\epsilon}
\newcommand{\f}{\varphi}
\newcommand{\bfphi}{{\boldsymbol{\f}}}
\renewcommand{\l}{\lambda}
\renewcommand{\k}{\kappa}
\newcommand{\lhat}{\hat\lambda}
\newcommand{\m}{\mu}
\newcommand{\bfmu}{{\boldsymbol{\mu}}}
\renewcommand{\o}{\omega}
\renewcommand{\r}{\rho}
\newcommand{\rbar}{{\bar\rho}}
\newcommand{\s}{\sigma}
\newcommand{\sbar}{{\bar\sigma}}
\renewcommand{\t}{\tau}
\newcommand{\z}{\zeta}

\newcommand{\D}{\Delta}
\newcommand{\G}{\Gamma}
\newcommand{\F}{\Phi}

%%%%%%%%%%%%%%%%%%%%
% Fraktur Alphabet %
%%%%%%%%%%%%%%%%%%%%
\newcommand{\ga}{{\mathfrak{a}}}
\newcommand{\gb}{{\mathfrak{b}}}
\newcommand{\gn}{{\mathfrak{n}}}
\newcommand{\gp}{{\mathfrak{p}}}
\newcommand{\gP}{{\mathfrak{P}}}
\newcommand{\gq}{{\mathfrak{q}}}

%%%%%%%%%%%%%%%%%%%
% Barred Alphabet %
%%%%%%%%%%%%%%%%%%%
\newcommand{\Abar}{{\bar A}}
\newcommand{\Ebar}{{\bar E}}
\newcommand{\Kbar}{{\bar K}}
\newcommand{\Pbar}{{\bar P}}
\newcommand{\Sbar}{{\bar S}}
\newcommand{\Tbar}{{\bar T}}
\newcommand{\ybar}{{\bar y}}
\newcommand{\phibar}{{\bar\f}}

%%%%%%%%%%%%%%%%%%%%%%%%%
% Calligraphic Alphabet %
%%%%%%%%%%%%%%%%%%%%%%%%%
\newcommand{\Acal}{{\mathcal A}}
\newcommand{\Bcal}{{\mathcal B}}
\newcommand{\Ccal}{{\mathcal C}}
\newcommand{\Dcal}{{\mathcal D}}
\newcommand{\Ecal}{{\mathcal E}}
\newcommand{\Fcal}{{\mathcal F}}
\newcommand{\Gcal}{{\mathcal G}}
\newcommand{\Hcal}{{\mathcal H}}
\newcommand{\Ical}{{\mathcal I}}
\newcommand{\Jcal}{{\mathcal J}}
\newcommand{\Kcal}{{\mathcal K}}
\newcommand{\Lcal}{{\mathcal L}}
\newcommand{\Mcal}{{\mathcal M}}
\newcommand{\Ncal}{{\mathcal N}}
\newcommand{\Ocal}{{\mathcal O}}
\newcommand{\Pcal}{{\mathcal P}}
\newcommand{\Qcal}{{\mathcal Q}}
\newcommand{\Rcal}{{\mathcal R}}
\newcommand{\Scal}{{\mathcal S}}
\newcommand{\Tcal}{{\mathcal T}}
\newcommand{\Ucal}{{\mathcal U}}
\newcommand{\Vcal}{{\mathcal V}}
\newcommand{\Wcal}{{\mathcal W}}
\newcommand{\Xcal}{{\mathcal X}}
\newcommand{\Ycal}{{\mathcal Y}}
\newcommand{\Zcal}{{\mathcal Z}}

%%%%%%%%%%%%%%%%%%%%%%%%%%%%
% Blackboard Bold Alphabet %
%%%%%%%%%%%%%%%%%%%%%%%%%%%%
\renewcommand{\AA}{\mathbb{A}}
\newcommand{\BB}{\mathbb{B}}
\newcommand{\CC}{\mathbb{C}}
\newcommand{\FF}{\mathbb{F}}
\newcommand{\GG}{\mathbb{G}}
\newcommand{\PP}{\mathbb{P}}
\newcommand{\NN}{\mathbb{N}}
\newcommand{\QQ}{\mathbb{Q}}
\newcommand{\RR}{\mathbb{R}}
\newcommand{\ZZ}{\mathbb{Z}}

%%%%%%%%%%%%%%%%%%%%%%%%%%
% Boldface Math Alphabet %
%%%%%%%%%%%%%%%%%%%%%%%%%%
\newcommand{\bfa}{{\mathbf a}}
\newcommand{\bfb}{{\mathbf b}}
\newcommand{\bfc}{{\mathbf c}}
\newcommand{\bfe}{{\mathbf e}}
\newcommand{\bff}{{\mathbf f}}
\newcommand{\bfg}{{\mathbf g}}
\newcommand{\bfp}{{\mathbf p}}
\newcommand{\bfr}{{\mathbf r}}
\newcommand{\bfs}{{\mathbf s}}
\newcommand{\bft}{{\mathbf t}}
\newcommand{\bfu}{{\mathbf u}}
\newcommand{\bfv}{{\mathbf v}}
\newcommand{\bfw}{{\mathbf w}}
\newcommand{\bfx}{{\mathbf x}}
\newcommand{\bfy}{{\mathbf y}}
\newcommand{\bfz}{{\mathbf z}}
\newcommand{\bfA}{{\mathbf A}}
\newcommand{\bfF}{{\mathbf F}}
\newcommand{\bfB}{{\mathbf B}}
\newcommand{\bfD}{{\mathbf D}}
\newcommand{\bfG}{{\mathbf G}}
\newcommand{\bfI}{{\mathbf I}}
\newcommand{\bfM}{{\mathbf M}}
\newcommand{\bfzero}{{\boldsymbol{0}}}

%%%%%%%%%%%%%%%%%%%%%%%%%%%%%%
% Miscellaneous New Commands %
%%%%%%%%%%%%%%%%%%%%%%%%%%%%%%
\newcommand{\Aut}{\operatorname{Aut}}
\newcommand{\Base}{\mathcal{B}} % indeterminacy (base) locus
\newcommand{\CanDiv}{\kappa} % Canonical divisor (K is already a number field)
\newcommand{\CM}{\operatorname{CM}}   % CM(O) = set of t with O_t = O
\newcommand{\codim}{\operatorname{codim}}
\newcommand{\Disc}{\operatorname{Disc}}
\newcommand{\Div}{\operatorname{Div}}
\newcommand{\Dom}{\operatorname{Dom}}
\newcommand{\Ell}{\operatorname{Ell}}   % Ell(O) = set of E with CM by O
\newcommand{\End}{\operatorname{End}}
\newcommand{\Fbar}{{\bar{F}}}
\newcommand{\Gal}{\operatorname{Gal}}
\newcommand{\GL}{\operatorname{GL}}
\newcommand{\Index}{\operatorname{Index}}
\newcommand{\Image}{\operatorname{Image}}
\newcommand{\liftable}{{\textup{liftable}}}
\newcommand{\hhat}{{\hat h}}
\newcommand{\Ker}{{\operatorname{ker}}}
\newcommand{\Lift}{\operatorname{Lift}}
\newcommand{\MOD}[1]{~(\textup{mod}~#1)}
\newcommand{\Mor}{\operatorname{Mor}}
\newcommand{\Norm}{{\operatorname{\mathsf{N}}}}
\newcommand{\notdivide}{\nmid}
\newcommand{\normalsubgroup}{\triangleleft}
\newcommand{\NotDom}{\operatorname{NotDom}}
\newcommand{\NS}{\operatorname{NS}}
\newcommand{\odd}{{\operatorname{odd}}}
\newcommand{\onto}{\twoheadrightarrow}
\newcommand{\ord}{\operatorname{ord}}
\newcommand{\Orbit}{\mathcal{O}}
\newcommand{\Per}{\operatorname{Per}}
\newcommand{\PrePer}{\operatorname{PrePer}}
\newcommand{\PGL}{\operatorname{PGL}}
\newcommand{\Pic}{\operatorname{Pic}}
\newcommand{\Prob}{\operatorname{Prob}}
\newcommand{\Qbar}{{\bar{\QQ}}}
\newcommand{\rank}{\operatorname{rank}}
\newcommand{\Rat}{\operatorname{Rat}}
\newcommand{\rel}{{\textup{rel}}}  %% relative
\newcommand{\red}{{\textup{red}}}  %% reduced
\newcommand{\Resultant}{\operatorname{Res}}
\renewcommand{\setminus}{\smallsetminus}
\newcommand{\shat}{{\hat s}}
\newcommand{\sing}{{\textup{sing}}} %% singular locus
\newcommand{\Span}{\operatorname{Span}}
\newcommand{\Spec}{\operatorname{Spec}}
\newcommand{\Support}{\operatorname{Support}}
\newcommand{\That}{{\hat T}}  %% blow-up of T to make P a morphism. 
                              %% maybe should call it V_P to indicate
                              %% dependence on P.
\newcommand{\tors}{{\textup{tors}}}
\newcommand{\Trace}{\operatorname{Trace}}
\newcommand{\tr}{{\textup{tr}}} % for K/k trace
\newcommand{\UHP}{{\mathfrak{h}}}    % Upper half plane
\newcommand{\<}{\langle}
\renewcommand{\>}{\rangle}

\newcommand{\ds}{\displaystyle}
\def\hw{\hidewidth}
\newcommand{\longhookrightarrow}{\lhook\joinrel\longrightarrow}
\newcommand{\longonto}{\relbar\joinrel\twoheadrightarrow}

%% autonumbered constants for use in inequalities
\newcount\ccount \ccount=0
\def\cc{\global\advance\ccount by1{c_{\the\ccount}}}

%%%%%%%%%%%%%%%%%%%%%%%%%%%%%%%%%%%%%%%%%%%%%%%%%%%%%%%%%%%%%%%%%%%%%%

\begin{abstract}
A primitive prime divisor of an element~$a_n$ of a
sequence~$(a_m)_{m\ge0}$ is a prime~$\gp$ that divides~$a_n$, but does
not divide~$a_m$ for all~$m<n$.  The Zsigmondy set~$\Zcal$ of the
sequence is the set of~$n$ such that~$a_n$ has no primitive prime
divisors. Let~$f:X\to X$ be a self-morphism of a variety, let~$D$ be
an effective divisor on~$X$, and let~$P\in X$, all defined
over~$\Qbar$.  We consider the Zsigmondy set~$\Zcal(X,f,P,D)$ of the
sequence defined by the arithmetic intersection of the $f$-orbit
of~$P$ with~$D$. Under various assumptions on~$X$,~$f$,~$D$, and~$P$,
we use Vojta's conjecture with truncated counting function to prove
that the set of points~$f^n(P)$ with~$n\in\Zcal(X,f,P,D)$ is not
Zariski dense in~$X$.
\end{abstract}

%% Non-TeX abstract (for ArXiv) 
%% A primitive prime divisor of an element a_n of a sequence (a_1,a_2,a_3,...) is a prime P that divides a_n, but does not divide a_m for all m < n.  The Zsigmondy set Z of the sequence is the set of n such that a_n has no primitive prime divisors. Let f : X --> X be a self-morphism of a variety, let D be an effective divisor on X, and let P be a point of X, all defined over the algebraic closure of Q. We consider the Zsigmondy set Z(X,f,P,D) of the sequence defined by the arithmetic intersection of the f-orbit of P with D. Under various assumptions on X, f, D, and P, we use Vojta's conjecture with truncated counting function to prove that the set of points f^n(P) with n in Z(X,f,P,D) is not Zariski dense in X.

\maketitle

%%%%%%%%%%%%%%%%%%%%%%%%%%%%%%%%%%%%%%%%%%%%%%%%%%%%%%%%%%%%%%%%%%%%%%
\section*{Introduction}
\label{section:introduction}
%%%%%%%%%%%%%%%%%%%%%%%%%%%%%%%%%%%%%%%%%%%%%%%%%%%%%%%%%%%%%%%%%%%%%%

Let $\Acal=(a_n)_{n\ge0}$ be a sequence of integers. A prime~$p$ is
called a \emph{primitive prime divisor for~$a_n$} if~$p\mid a_n$, but
$p\nmid a_m$ for all $0\le m<n$ with $a_m\ne0$.\footnote{The classical
  definition does not include the requirement that $a_m\ne0$, but
  including this condition allows for cleaner statements of
  theorems. Note that if some~$a_m=0$, then~$a_m$ is divisible by
  every prime, so if we didn't exclude such~$a_m$, then~$a_n$ with
  $n>m$ would never have a primitive prime divisor.}  The existence of
primitive prime divisors in various types of sequences has been much
studied, both for its intrinsic interest and for numerous
applications. 
\par
The \emph{Zsigmondy set of~$\Acal$} is
\[
  \Zcal(\Acal) = \{n\ge0 : \text{$a_n$ has no primitive prime divisors}\}.
\]
A typical result, due to Zsigmondy~\cite{zsigmondy}, is that if
$u>v>0$ are coprime integers, then the Zsigmondy set of the sequence
$(u^n-v^n)_{n\ge0}$ is finite, and indeed
$\Zcal\bigl(u^n-v^n\bigr)\subseteq\{1,2,6\}$.
\par
In this note we investigate primitive prime divisors in sequences
associated to dynamical systems. 
In this setting, we will use a strong version of a conjecture of Paul
Vojta (Conjecture~\ref{conjecture:vojtatrunc}) to show that certain
Zsigmondy sets are small.  We start by stating a special case of our
main theorem that does not require any technical definitions or
notation.  In the statement, we assume that points in~$\PP^N(\QQ)$ are
always written in normalized form, i.e., as
\[
  [x_0,\ldots,x_N]\quad\text{with}\quad x_0,\ldots,x_N\in\ZZ
  \quad\text{and}\quad \gcd(x_0,\ldots,x_N)=1.
\]
\begin{theorem}
\label{theorem:mainthmoverQ}
Assume that Vojta's conjecture with truncated counting
function~\cite[Conjecture~22.5]{MR2757629} is true for~$\PP^N$.  Let
$f:\PP^N\to\PP^N$ be a morphism of degree~$d\ge3$ defined over~$\QQ$.
Let~$F(X_0,\ldots,X_N)\in\ZZ[X_0,\ldots,X_N]$ be a non-constant
homogeneous polynomial such that the locus $F=0$ in~$\PP^N$ is a
reduced normal-crossings hypersurface. Assume further that
\begin{equation}
  \label{eqn:degFbig}
  \deg(F) > (N+1)\left(1+\frac{1}{d-2}\right).
\end{equation}
Let $P\in\PP^N(\QQ)$, and let
\[
  \Pcal = \bigl( F(f^n(P)) \bigr)_{n\ge0} \subset \ZZ
\]
be the sequence of values of~$F$ on the points in the~$f$-orbit
of~$P$. Then the set 
\[
  \bigl\{f^n(P):n\in\Zcal(\Pcal)\bigr\}
\]
is not Zariski dense in~$\PP^N$.
\end{theorem}

Primitive divisors and Zsigmondy sets for dynamically defined
sequences have been studied by many mathematicians.  In
Section~\ref{section:history} we briefly recount some of this history,
which started with the $19^{\text{th}}$-century work of
Bang~\cite{Bang1886} and Zsigmondy~\cite{zsigmondy}.  We mention in
particular a recent paper by Gratton, Nguyen, and
Tucker~\cite{arxiv1208.2989} in which they prove a strengthened
version of Theorem~\ref{theorem:mainthmoverQ} for~$\PP^1$ under the
assumption that the ABC~conjecture of Masser and Oesterl\'e is true.
Their paper, as well as a paper of Yasufuku~\cite{MR2868570} in which
he uses Vojta's conjecture to study integrality of points in orbits,
were the inspirations for the present note, since Vojta's conjecture
(with truncated counting function) is a natural higher-dimensional
analogue of the ABC~conjecture.

In order to describe our general version of
Theorem~\ref{theorem:mainthmoverQ}, we reformulate it geometrically.
The zero-locus of the polynomial~$F$ defines an ample effective
divisor~$D$, and for most primes~$p$, the condition that~$p$
divide~$F\bigl(f^n(P)\bigr)$ is equivalent to the condition that when
we reduce modulo~$p$, we have
\[
  \widetilde{f^n(P)}\bmod p \in \tilde{D}\bmod p,
\]
i.e., the reduction of the point~$f^n(P)$ modulo~$p$ lies on the
reduction of the divisor~$D$ modulo~$p$. So~$p$ is a primitive divisor
of~$F\bigl(f^n(P)\bigr)$ if and only if
\[
  \widetilde{f^n(P)}\bmod p \in \tilde{D}\bmod p
\]
and
\[
  \widetilde{f^m(P)}\bmod p \notin \tilde{D}\bmod p
  \quad\text{for all $1\le m<n$.}
\]
With this formulation, we can define Zsigmondy sets relative to
effective divisors for orbits of points on arbitrary varieties defined
over arbitrary global fields, although there will be a small amount of
ambiguity due to the necessity of choosing an integral model. However,
different models lead to Zsigmondy sets that differ in finitely many
elements, so our finiteness statements will be independent of the
chosen model.

We now state a special case of our main result, using notation that is
described in detail in
Sections~\ref{section:notation},~\ref{section:vojtaconj},
and~\ref{section:ynsystems}.  For a generalization and proof of
Theorem~\ref{theorem:mainthmspcase}, see Theorem~\ref{theorem:mainthm}
and Remark~\ref{remark:mainthmj0} in
Section~\ref{section:primdivsandvojta}.

\begin{theorem}
\label{theorem:mainthmspcase}
Let $(X_K,f_K,H)$ be a polarized dynamical system of degree~$d\ge3$
with~$\rank\NS(X_K)=1$, let~$P\in X_K(K)$ with $\hhat_{f,H}(P)>0$, and
let~$D_K$ be an ample effective reduced normal crossings divisor
on~$X_K$ satisfying
\begin{equation}
  \label{eqn:degKbig}  \deg(D_K) > \frac{d-1}{d-2}\cdot \deg(-\CanDiv_X).
\end{equation}
Further, assume that Vojta's conjecture
\textup(Conjecture~\textup{\ref{conjecture:vojtatrunc})} is true
for the variety~$X_K$ and the divisor~$D_K$. Then
\begin{equation}
  \label{eqn:fnPnZOfP}
  \Bigl\{ f^n(P) : n\in \Zcal\bigl( \Orbit_f(P), D \bigr) \Bigr\}
\end{equation}
is not Zariski dense in~$X_K$. More precisely, except for finitely
many points, the set~\eqref{eqn:fnPnZOfP} is contained in a proper
closed subvariety~$Y\subsetneq X$ that does not depend on~$P$.
\end{theorem}

\begin{remark}
We comment briefly on the conditions~\eqref{eqn:degFbig}
and~\eqref{eqn:degKbig} which require~$\deg(F)$ and $\deg(D_K)$ to be
sufficiently large. In some sense, these conditions are necessary,
since for example the theorem is false on~$\PP^1$ with $\deg F=2$
if we take $F(X,Y)=XY$ and $f\bigl([X,Y]\bigr)=[X^d,Y^d]$. The
authors of~\cite{arxiv1208.2989} eliminate this problem on~$\PP^1$ by
replacing~$F$ with~$F\circ f^j$ for some fixed~$j$, thereby increasing
the degree.  This works provided that the function~$F\circ f^j$
acquires enough roots, which will be true unless~$F$ and~$f$ are of a
very special form. In our higher dimensional setting, we can do the
same thing, but we do not have a good description of the ``bad''
$(f,F)$ pairs.  Thus in Theorem~\ref{theorem:mainthmoverQ} we can
replace~\eqref{eqn:degFbig} with the condition that there exists
some~$j\ge0$ such that $\{ F\circ f^j=0\}$ is a reduced normal
crossings hypersurface and~$F$ satisies
\[
  \deg F > \frac{d-1}{d-2}\cdot \frac{N+1}{d^j}.
\]
However, we do not know a nice characterization of the pairs~$(F,f)$
that ensure the existence of such a~$j$.
\end{remark}

In the final two sections we consider related problems and
generalizations.  Section~\ref{section:primdivalggps} discusses the
problem of primitive divisors in algebraic groups, where iteration of
a map may be replaced by the powering map to create sequences that are
denser, and hence less likely to have primitive divisors. In this
situation, even assuming Vojta's conjecture, we explain why we are
unable to prove anything significant about the Zsigmondy set.  Finally,
in Section~\ref{section:primcycles}, we replace points and
divisors, which have dimension~$1$ and codimension~$1$, respectively, with
arbitrary subvarieties that have (arithmetically) complementary
codimensions.  Again, even assuming Vojta's conjecture, we have no
theorems, but we take the opportunity to describe a general framework
and to raise some questions.

%%%%%%%%%%%%%%%%%%%%%%%%%%%%%%%%%%%%%%%%%%%%%%%%%%%%%%%%%%%%%%%%%%%%%%
\section{A Brief History of Primitive Prime Divisors}
\label{section:history}
%%%%%%%%%%%%%%%%%%%%%%%%%%%%%%%%%%%%%%%%%%%%%%%%%%%%%%%%%%%%%%%%%%%%%%

As noted earlier, the history of primitive divisors started in the
$19^{\text{th}}$ century with work of Bang~\cite{Bang1886} and
Zsigmondy~\cite{zsigmondy}, who studied the Zsigmondy set
of~$u^n-v^n$. The Bang--Zsigmondy sequences are examples of
\emph{divisibility sequences}, that is, sequences $(a_n)_{n\ge1}$ such
that $a_n\mid a_{nd}$ for all $n,d\ge1$.  They are associated to the
multiplicative group~$\GG_m(\QQ)$.  Lucas sequences are
generalizations attached to twisted multiplicative groups. The study
of primitive divisors in Lucas sequence was completed in~2001 by Bilu,
Hanrot, and Voutier~\cite{MR1863855}, who proved that a Lucas sequence
has primitive divisors for each term with~$n>30$. The existence of
primitive divisors has many applications, for example in the original
proof of Wedderburn's theorem.

Elliptic divisibility sequences (EDS) are divisibility sequences with
$\GG_m$ replaced by an elliptic curve. The properties of EDS were
first studied formally by Ward~\cite{MR0023275}.  The
author~\cite{MR961918} gave an ineffective proof of the existence of
primitive divisors in~EDS, and there are many papers giving
quantitative and/or effective estimates; see for
example~\cite{MR2605536,MR2867920,MR2377368,MR2853048}.
Poonen~\cite{MR2041072} used EDS primitive divisors to resolve certain
cases of Hilbert's $10^{\text{th}}$ problem in number fields.

More generally, one can define divisibility sequences associated to a
non-torsion point~$P$ in any (commutative) algebraic group~$G$. In
Section~\ref{section:primdivalggps} we discuss this general
construction, which was originally described
in~\cite[Section~6]{MR2162351}.

One can also define divisibility sequences in a dynamical
setting.  For example, let $f(x)\in\QQ(x)$ be a rational function,
let~$\a\in\QQ$ be a wandering point, and let~$\b\in\QQ$ be a periodic
point. Then the sequence of numerators of~$f^n(\a)-\b$ satisfies a
strong divisibility property, and subject to appropriate conditions
on~$f$, Ingram and the author~\cite{arXiv:0707.2505} proved that the
Zsigmondy set of this sequence is finite. 
This was applied by Voloch
and the author~\cite{MR2496466} to prove a local-global criterion for
dynamics on~$\PP^1$.

The situation regarding dynamical primitive divisors becomes much more
difficult if the target point~$\b$ is a wandering point, since the
resulting sequence has no obvious divisibility properties. This was
left as an open question in~\cite{arXiv:0707.2505}, and even with the
assumption that the ABC~conjecture is true, it required considerable
ingenuity by Gratton, Nguyen, and Tucker~\cite{arxiv1208.2989} to
prove the existence of primitive divisors for dynamical systems
on~$\PP^1$ with wandering target point. The present paper
replaces~$\PP^1$ with an arbitrary variety and replaces the wandering
target point with a wandering target divisor satisfying some geometric
conditions. And in order to prove anything in this higher dimensional
setting, we require a strong form of Vojta's conjecture in place of
the ABC~conjecture.

%%%%%%%%%%%%%%%%%%%%%%%%%%%%%%%%%%%%%%%%%%%%%%%%%%%%%%%%%%%%%%%%%%%%%%
\section{Primitive Prime Divisors and Zsigmondy Sets}
\label{section:notation}
%%%%%%%%%%%%%%%%%%%%%%%%%%%%%%%%%%%%%%%%%%%%%%%%%%%%%%%%%%%%%%%%%%%%%%

We begin by setting some notation that will remain fixed throughout
this article. We start with our basic objects of study.

\begin{notation}
\item[$K$]
a number field or a characteristic~$0$ one-dimensional function field.
\item[$X_K$]
a smooth projective variety defined over~$K$.
\item[$D_K$]
an effective divisor on~$X_K$.
\end{notation}

In order to define primitive prime divisors and Zsigmondy sets, we
first choose an integral model for~$X_K$. 

\begin{notation}
\item[$R$]
a Dedekind domain in~$K$ whose fraction field is~$K$.
\item[$X_R$]
a scheme that is smooth and proper over $\Spec R$ and with generic
fiber~$X_K$, i.e.,  satisfying $X_R\times_RK=X_K$.
\item[$D_R$]
the closure of~$D_K$ in~$X_R$.
\end{notation}

We call~$X_R$ a \emph{model of~$X_K$}.  We note that~$X_K$ always has
a model, since~$X_K$ is smooth and projective over~$K$, so we can
take~$R$ to be the ring of~$S$-integers in~$K$ for a sufficiently
large finite set of places~$S$.  Of course, the variety~$X_K$ has many
different models, but as we will see, the associated Zsigmondy sets
are equal up to finite alteration.

Since~$X_R$ is proper over~$\Spec R$, a point~$Q\in X_K(K)$ defines a
point $Q_R\in X_R(R)$, i.e., the point~$Q$ induces a morphism 
\[
  \s_Q:\Spec R\to X_R.
\]
We write~$|D_K|$ for the support of~$D_K$.  If~$Q\notin|D_K|$, then
$\s_Q^*(D_R)$ is an effective divisor on~$\Spec R$, which is simply a
formal sum of points in~$\Spec{R}$,
\[
  \s_Q^*(D_R) = \sum_{\gp\in\Spec R} n_\gp(Q,D)\gp.
\]

\begin{definition}
The (\emph{global}) \emph{arithmetic intersection ideal} of~$Q$
and~$D$ is the ideal
\begin{equation}
  \label{eqn:globintid}
  (Q\cdot D)_R = \prod_{\gp\in\Spec R} \gp^{n_\gp(Q,D)}.
\end{equation}
\end{definition}

\begin{remark}
In the notation of Theorem~\ref{theorem:mainthmoverQ}, if we let
$D_K=\{F=0\}\in\Div(\PP^N_K)$, then the ideal $(Q\cdot D)_\ZZ$ is
the ideal in~$\ZZ$ generated by the integer~$F(Q)$.
\end{remark}

\begin{definition}
Let~$\Qcal=\{Q_n\}_{n\ge0}$ be a sequence of points in
$X(K)\setminus|D_K|$.  We say that~$\gp\in\Spec R$ is a
\emph{primitive \textup(prime\textup) divisor of~$Q_n$ relative to~$D$} if
\[
  \gp \mid (Q_n\cdot D)_R
  \qquad\text{and}\qquad
  \gp \nmid (Q_m\cdot D)_R
  \quad\text{for all $0\le m<n$.}
\]
The \emph{Zsigmondy set of~$\Qcal$ relative to~$D$} is
\[
  \Zcal(\Qcal,D) 
  = \{n\ge0 : \text{$Q_n$ has no primitive divisors relative to $D$}\}.
\]
\end{definition}

The definition of primitive divisors and the Zsigmondy set clearly
depend on our choice of the model~$X_R$. We now show
that~$\Zcal(\Qcal,D)$ is well-defined up to alteration by finitely
many elements.  This shows in particular that the statement of
Theorem~\ref{theorem:mainthm}, which is our main result, is independent
of the choice of a model for~$X_K$.

\begin{proposition}
\label{proposition:Zmodelindep}
Let~$\Zcal=\Zcal(\Qcal,D)$ and~$\Zcal'=\Zcal'(\Qcal,D)$ be Zsigmondy
sets of~$\Qcal$ relative to~$D$ that are defined using models~$X_R$
and~$X'_R$ of~$X_K$, say associated to rings~$R$ and~$R'$, respectively.
Then the set difference
\begin{equation}
  \label{eqn:ZcupsetminusZcap}
  (\Zcal\cup\Zcal') \setminus (\Zcal\cap\Zcal')
  \quad\text{is finite.} 
\end{equation}
\end{proposition}
\begin{proof}
Each point in~$\Spec R$ corresponds to a distinct valuation of~$K$,
and all but finitely many (equivalence classes) of valuations of~$K$ come
from $\Spec R$. The same is true of~$\Spec R'$, so we can find
a model~$X_R''$ for~$X_K$ over a ring~$R''$ such that
\[
  \begin{array}{ccc}
    X_R'' & \longhookrightarrow & X_R \\
    \Big\downarrow && \Big\downarrow  \\
    \Spec(R'') & \longhookrightarrow & \Spec(R) \\
  \end{array}
  \quad\text{and}\quad
  \begin{array}{ccc}
    X_R'' & \longhookrightarrow & X_R' \\
    \Big\downarrow && \Big\downarrow  \\
    \Spec(R'') & \longhookrightarrow & \Spec(R') \\
  \end{array}.
\]
\par
By definition, for each non-negative integer $n\notin\Zcal$, we can
find a prime~$\gp\in\Spec R$ that is a primitive divisor of~$Q_n$.
Since each prime in~$\Spec R$ is a primitive divisor for at most
one~$Q_n$, and since $\Spec(R)\setminus\Spec(R'')$ is finite, we can
find an~$N$ so that for every $n\ge N$ with $n\notin\Zcal$, there is a
prime~$\gp\in\Spec(R'')$ that is a primitive divisor of~$Q_n$.
\par
Similarly, we can find an~$N'$ so that for every $n\ge N'$ with
$n\notin\Zcal'$, there is a prime~$\gp\in\Spec(R'')$ that is a
primitive divisor of~$Q_n$. So if we let~$\Zcal''$ be the Zsigmondy
set of~$\Qcal$ relative to~$D$ using the model~$X_R''$
and if we set $N''=\max\{N,N'\}$, then we have proven that
\[
  \{n\in \Zcal'' : n\ge N''\}
  =   \{n\in \Zcal' : n\ge N''\}
  =   \{n\in \Zcal : n\ge N''\}.
\]
This gives~\eqref{eqn:ZcupsetminusZcap}, which completes the proof of
Proposition~\ref{proposition:Zmodelindep}.
\end{proof}

%%%%%%%%%%%%%%%%%%%%%%%%%%%%%%%%%%%%%%%%%%%%%%%%%%%%%%%%%%%%%%%%%%%%%%
\section{Height Functions and Vojta's Conjecture}
\label{section:vojtaconj}
%%%%%%%%%%%%%%%%%%%%%%%%%%%%%%%%%%%%%%%%%%%%%%%%%%%%%%%%%%%%%%%%%%%%%%

In this section we discuss height functions and set the notation
required to state Vojta's conjecture. 

\begin{notation}
\item[$M_K$]
a complete set of inequivalent absolute values on~$K$. We
write~$M_K^0$, respectively~$M_K^\infty$, for the set of
non-archimedean, respectively archimedean, absolute values in~$M_K$.
\item[$N_\gp$]
the number of elements in the residue field associated to a place
$\gp\in M_K^0$.
\item[$h_D$]
a Weil height on~$X_K$ for the divisor~$D$
relative to the field~$K$.
\item[$\l_{D,v}$]
a $v$-adic local height on~$(X\setminus|D|)(K_v)$ for the divisor~$D$,
relative to the field~$K_v$. If~$D$ is effective and~$v\in M_K^0$, we
assume that~$\l_{D,v}$ is chosen to be a non-negative function.
\end{notation}

\begin{remark}
\label{remark:deflDgpQ}
Having fixed a model~$X_R$, we can use intersection theory to specify a
local height~$\l_{D,\gp}$ for $\gp\in\Spec R\subset M_K^0$ via the formula
\begin{equation}
  \label{eqn:lDpQordp}
  \l_{D,\gp}(Q) = \ord_\gp (Q\cdot D)_R \cdot \log N_\gp,
\end{equation}
where we recall that~$(Q\cdot D)_R$ is the arithmetic intersection
ideal~\eqref{eqn:globintid}.  (Cf.~\cite[Chapter 11, Section
  5]{lang:diophantinegeometry}.)  We will assume henceforth that
for~$\gp\in\Spec R$, the local height~$\l_{D,\gp}$ has been chosen in
this way.
\end{remark}

We further assume that our absolute values and local and global
heights are chosen to satisfy
\[
  h_D(P) = \sum_{v\in M_K} \l_{D,v}(P)
  \quad\text{for all $P\in X_K(K)\setminus|D|$.}
\]
For basic material on height functions and their normalizations, see
for example~\cite[\S\S B.1--B.3]{hindrysilverman:diophantinegeometry}
or~\cite[Chapters~3 and~4]{lang:diophantinegeometry}.

\begin{definition}
An effective divisor~$D\in\Div(X)$ is a \emph{normal crossings
  divisor} if $D=\sum_{i=1}^r n_iD_i$, where the~$D_i$ are irreducible
subvarieties and the variety $\bigcup_{i=1}^r|D_i|$ has normal
crossings.  If all of the~$n_i=1$, then~$D$ is a \emph{reduced
  divisor}.  We write~$D^\red=\sum_{i=1}^rD_i$ for the reduced divisor
associated to~$D$.
\end{definition}

Vojta's original conjectures~\cite{MR883451} used arithmetic analogues
of the proximity and counting functions of classical Nevanlina theory.
Later, Vojta took stronger results in Nevanlinna theory
involving truncated counting functions and transplanted them into an
arithmetic setting. This led to arithmetic conjectures that are
natural generalizations of the ABC-conjecture of Masser and
Oesterl\'e~\cite{MR992208} and of Szpiro's
conjecture~\cite{MR1065151}.

\begin{definition}
Let $S\subset M_K$ be a finite set of places containing $M_K^\infty$.
The (\emph{arithmetic}) \emph{truncated counting function} is
\[
  N_S^{(1)}(D,P)
  = \sum_{\gp\notin S} \min\bigl\{\l_{D,\gp}(P),\log N_\gp \bigr\}
  = \sum_{\gp\notin S,\; n_\gp(P,D)\ge1}\log N_\gp.
%%  = \sum_{\substack{\gp\notin S\\ n_\gp(P,D)\ge1\\}}\log N_\gp.
\]
It is well-defined for $P\in X(K)\setminus|D|$.
\end{definition}

\begin{remark}
It is clear from the definition that
\[
  N_S^{(1)}(D,P)=N_S^{(1)}(D^\red,P),
\]
since $\l_{D,\gp}(P)=0$ if and only if $\l_{D^\red,\gp}(P)=0$.
\end{remark}

\begin{conjecture}[Vojta's conjecture with truncated counting 
function {\cite[Conjecture~22.5]{MR2757629}}]
\label{conjecture:vojtatrunc}
Let $S\subset M_K$ be a finite set
of places containing~$M_K^\infty$, let~$D$ be a normal crossings
divisor on~$X$, let~$\CanDiv_X$ be a canonical divisor on~$X$, and
let~$H$ be an ample line bundle on~$X$.  All heights are taken
relative to the number field~$K$.
%% Let $\Sigma\subset X(K)\setminus|D|$ have the property that
%% every infinite subset of~$\Sigma$ is Zariski dense in~$X$. \textup(One
%% says that~$\Sigma$ is \emph{generic in~$X$}.\textup) Then
%% \[
%%   N_S^{(1)}(D,P)
%%   \ge h_{\CanDiv_X+D}(P) - O\left(\sqrt{h_H(P)}+1\right)
%%   \quad\text{for all $P\in\Sigma$,}
%% \]
%% where the big-$O$ constant is independent of~$P$, but may depend
%% on~$X,S,D,\Sigma,H$ and the choice of height functions.
\par
For all $\e>0$ there is a proper Zariski closed set
\[
  Y=Y(X,D,H,\e)\subset X
\]
such that for all~$C\ge0$, the inequality
\[
  N_S^{(1)}(D,P)
  \ge h_{\CanDiv_X+D}(P) - \e h_H(P) - C
\]
holds for all but finitely many $P\in (X\setminus Y)(K)$
\end{conjecture}

%%%%%%%%%%%%%%%%%%%%%%%%%%%%%%%%%%%%%%%%%%%%%%%%%%%%%%%%%%%%%%%%%%%%%%
\section{Dynamical Systems}
\label{section:ynsystems}
%%%%%%%%%%%%%%%%%%%%%%%%%%%%%%%%%%%%%%%%%%%%%%%%%%%%%%%%%%%%%%%%%%%%%%

We study the dynamical system given by iteration of the $K$-mor\-phism
\[
  f_K : X_K \longrightarrow X_K.
\]
We denote the~$f_K$-orbit of a point~$P\in X_K(K)$ by
\[
  \Orbit_f(P)
  = \bigl\{P,f_K(P), f_K^2(P),f_K^3(P),\ldots\bigr\}.
\]
\par
For the remainder of this article, we assume without further comment
that the Dedekind domain~$R$ and model~$X_R$ have been chosen so
that~$f_K$ extends to an~$R$-scheme morphism
\[
  f_R : X_R \longrightarrow X_R.
\]
We call $(X_R,f_R)$ a model for $(X_K,f_K)$.  Having normalized our
local heights using intersection theory on~$X_R$ via
formula~\eqref{eqn:lDpQordp}, we see that they behave functorially,
\begin{equation}
  \label{eqn:lDpfPeqlfDpP}
  \l_{D,\gp}\bigl(f_K(P)\bigr) = \l_{f_K^*D,\gp}(P).
\end{equation}
This follows from the projection formula
\[
  (f_K(Q),D)_R = (Q,f_K^*D)_R,
\]
which is the arithmetic analogue of~\cite[Appendix~A,
  Formula~A4]{hartshorne}.
\par
Using the projection formula~\eqref{eqn:lDpfPeqlfDpP} in the
definition of the truncated counting function yields
\[
  N_S^{(1)}\bigl(D,f_K(P)\bigr)=N_S^{(1)}(f_K^*D,P).
\]
More generally, for any~$j\ge0$ we have
\begin{equation}
  \label{eqn:NDfjPeqNfjDredP}
  N_S^{(1)}\bigl(D,f_K^j(P)\bigr)
  = N_S^{(1)}\bigl(((f_K^j)^*D)^\red,P\bigr).
\end{equation}

\begin{definition}
A \emph{polarized dynamical system} (\emph{defined over~$K$}) is a
triple $(X_K,f_K,H)$ consisting of a smooth projective variety~$X_K$, a
$K$-mor\-phism \text{$f_K:X_K\to X_K$}, and an ample
divisor~$H\in\Div(X_K)\otimes\RR$ with the property that
\[
  f_K^*H \sim d H
  \quad\text{for some $d>1$.}
\]
Here~$\sim$ indicates linear equivalence in~$\Pic(X)\otimes\RR$.
We call~$d$ the \emph{degree of the polarized dynamical system}.
\end{definition}

\begin{example}
Let~$f_K:\PP_K^N\to\PP_K^N$ be a morphism of degree at least two, and
let~$H$ be any ample divisor on~$\PP_K^N$. Then~$(\PP_K^N,f_K,H)$ is a
polarized dynamical system of degree equal to~$\deg(f)$.
\end{example}

The following well-known result summarizes the theory of canonical
heights for polarized dynamical systems.

\begin{theorem}
\label{theorem:canht}
Let $(X_K,f_K,H)$ be a polarized dynamical system of degree~$d$.
Then for all $P\in X(K)$, the limit
\[
  \hhat_{f,H}(P) = \lim_{n\to\infty} d^{-n}h_H\bigl(f_K^n(P)\bigr)
\]
exists and has the following properties\textup:
\begin{parts}
\Part{(a)}
$\hhat_{f,H}(P) = h_H(P) + O(1)$, where the~$O(1)$ depends
on~$X_K$,~$f_K$, and the choice of height function~$h_H$, but is
independent of~$P$.
\Part{(b)}
$\hhat_{f,H}\bigl(f_K(P)\bigr)=d\hhat_{f,H}(P)$.
\Part{(c)}
If~$K$ is a number field, then $\hhat_{f,H}(P)=0$ if and only if~$P$
is preperiodic, i.e., has finite forward orbit.
\end{parts}
\end{theorem}
\begin{proof}
See for example~\cite{callsilv:htonvariety}
or~\cite[\S3.4]{MR2316407}.  We note that~(c) is also true over
function fields if~$f_K$ is not isotrivial, but the proof is much more
difficult; see~\cite{arxiv0601046,arxiv0510444}.
\end{proof}

\begin{definition}
If the N\'eron-Severi group~$\NS(X_K)$ of the variety~$X_K$ has
rank~$1$, then we fix a \emph{degree map}
\[
  \deg:\NS(X_K)\otimes\RR \xrightarrow{\;\sim\;}\RR
\]
with the property that $\deg(H)>0$ for ample divisors~$H$.  The degree
map is well-defined up to multiplication by a positive constant.
\end{definition}

Height functions have many functorial properties; see
for example~\cite[Theorem~B.3.2]{hindrysilverman:diophantinegeometry}.
We will need the following special case of one of these properties.

\begin{proposition}
\label{proposition:htforalgeq}
Assume that~$\NS(X_K)$ has rank~$1$, let~$H\in\Div(X_K)$ be
an ample divisor, and let~$D\in\Div(X_K)$ be an arbitrary divisor.
Then for every~$\e>0$ we have
\[
  \bigl|(\deg H)h_D-(\deg D)h_H\bigr| \le \e h_H + O(1).
\]
\end{proposition}
\begin{proof}
The assumption on~$\NS(X_K)$ implies that the divisor class of $(\deg
H)D$ is algebraically equivalent to the divisor class of $(\deg D)H$
in~$\NS(X)\otimes\RR$, since they both have the same degree. We let
$E=(\deg H)D-(\deg D)H$ and
apply~\cite[Theorem~B.3.2]{hindrysilverman:diophantinegeometry}(f),
which says that since~$E$ is algebraically equivalent to~$0$, we have
$h_E=o(h_H)$.
\end{proof}

%%%%%%%%%%%%%%%%%%%%%%%%%%%%%%%%%%%%%%%%%%%%%%%%%%%%%%%%%%%%%%%%%%%%%%
\section{Primitive Divisors in Dynamical Systems and Vojta's Conjecture}
\label{section:primdivsandvojta}
%%%%%%%%%%%%%%%%%%%%%%%%%%%%%%%%%%%%%%%%%%%%%%%%%%%%%%%%%%%%%%%%%%%%%%

In this section we use Vojta's conjecture to prove our main theorem,
which says that certain dynamically defined Zsigmondy sets are small.

\begin{theorem}
\label{theorem:mainthm}
Let $(X_K,f_K,H)$ be a polarized dynamical system of degree~$d$, and
let~$D_K$ be an ample effective divisor on~$X_K$. Make the following
assumptions:
\begin{parts}
\Part{\textbullet}
$\NS(X_K)$ has rank~$1$.
\Part{\textbullet}
$P\in X(K)$ is a point with $\hhat_{f,H}(P)>0$.
\Part{\textbullet}
There is an integer~$j\ge0$ such that the divisor
\[
  \D_j :=  \bigl((f^j)^*D\bigr)^\red
\]
is a normal crossings divisor whose degree satisfies
\begin{equation}
  \label{eqn:degDjdjgt}
  \frac{\deg \D_j}{d^j} > \frac{\deg D}{d-1} + \frac{\deg(-\CanDiv_X)}{d^j}.
\end{equation}
\Part{\textbullet}
Vojta's conjecture
\textup(Conjecture~\textup{\ref{conjecture:vojtatrunc})} is true for
the variety~$X$ and the divisor $\bigl((f^j)^*D\bigr)^\red$.
We let~$Y_{f,D,j}\subsetneq X$ denote the exceptional set in
Vojta's conjecture.
\end{parts}
Then the Zsigmondy set of the $f$-orbit of~$P$ relative to~$D$ has
the property that
\[
  \Bigl\{ f^n(P) : n \in \Zcal\bigl( \Orbit_f(P), D ) \Bigr\}
  \subset Y_{f,D} \cup \{\text{finite set}\}.
\]
\end{theorem}

\begin{remark}
\label{remark:mainthmj0}
Taking~$j=0$ in Theorem~\ref{theorem:mainthm},
condition~\eqref{eqn:degDjdjgt} reduces to the assumption that~$D$
itself is a reduced normal crossings divisor satisfying
\[
  \deg D > \frac{d-1}{d-2}\cdot \deg(-\CanDiv_X).
\]
This is the version of Theorem~\ref{theorem:mainthm} that
we stated in the introduction.
\end{remark}

\begin{remark}
If the divisor $(f^j)^*D$ in Theorem~\ref{theorem:mainthm} is
already reduced, then in~$\NS(X)\otimes\RR$ we have
\[
  \D_j = \bigl((f^j)^*D\bigr)^\red
  = (f^j)^*D
  \equiv \frac{\deg D}{\deg H}(f^j)^*H
  \equiv \frac{\deg D}{\deg H}d^jH
  \equiv d^jD.
\]
After a little bit of algebra, the degree
condition~\eqref{eqn:degDjdjgt} then becomes
\[
  \deg(D) > \frac{d-1}{d-2}\cdot \frac{\deg(-\CanDiv_X)}{d^j}.
\]
Hence in  Theorem~\ref{theorem:mainthm} 
it suffices to assume that~$(f^j)^*D$ is reduced for some
\[
  j > \log_d\left(\frac{d-1}{d-2}\cdot \frac{\deg(-\CanDiv_X)}{\deg(D)}\right).
\]
\end{remark}

\begin{remark}
It would be interesting to weaken some of the hypotheses in
Theorem~\ref{theorem:mainthm}. For example, what happens if~$\NS(X_K)$
has rank larger than~$1$?  Or we could take a dynamical system that is
not quite polarized. For example, suppose that~$f_K$ is an
automorphism, and suppose that there are nef divisors~$D^+$ and~$D^-$
that are eigendivisors for~$f_K$ and~$f_K^{-1}$, respectively, such
that~$D^++D^-$ is ample. This is the situation for certain~K3
surfaces~\cite{silverman:K3heights}. Or we could take an affine
morphism~$f:\AA^N\to\AA^N$ that does not extend to a morphism
of~$\PP^N$.  In this setting, we mention that if~$f$ is a regular
automorphism, then there is a good theory of canonical
heights~\cite{arxiv0909.3573,arxiv0909.3107} that might be useful.
\end{remark}

\begin{proof}[Proof of Theorem~$\ref{theorem:mainthm}$]
To ease notation, we write~$f$ for~$f_K$, and having chosen an ample
divisor~$H$, we normalize our degree function to
satisfy~$\deg(H)=1$. We fix a model~$X_R$ for~$X_K$, and when applying
Vojta's conjecture, we take~$S$ to be the finite set of places
\[
  S = M_K \setminus \Spec(R).
\]
We observe that a prime~$\gp\in\Spec R$ is a primitive divisor
for~$f^n(P)$ relative to~$D$ if and only if
\[
  \ord_\gp\bigl(f^n(P),D)_R>0,
  \text{ and }
  \ord_\gp\bigl(f^m(P),D)_R=0
  \quad\text{for all $0\le m<n$.}
\]
In view of the definition 
\[
  \l_{D,\gp}(Q) = \ord_\gp (Q\cdot D)_R \cdot \log N_\gp
\]
of our local height functions (Remark~\ref{remark:deflDgpQ}), we see
that~$f^n(P)$ will have a primitive  divisor relative to~$D$ if
the following sum is strictly positive:
\begin{equation}
  \label{eqn:issumposq}
  B_n :=
  \sum_{\substack{\gp\in\Spec R\\ \hw\l_{D,\gp}(f^mP)=0 \text{ for all $m<n$}\hw\\  }}
    \l_{D,\gp}(f^nP).
\end{equation}
We take a point~$f^n(P)\notin Y_{f,D,j}$ which is not one of the
finitely many exceptions in Vojta's conjecture
(Conjecture~\ref{conjecture:vojtatrunc}) and compute
\begin{align*}
  B_n
  &\ge \sum_{\substack{\gp\in\Spec R\\ \l_{D,\gp}(f^nP) > 0\\
        \hw\l_{D,\gp}(f^mP)=0 \text{ for all $m<n$}\hw\\  }} \log N_\gp \\
  &\ge \sum_{\substack{\gp\in\Spec R\\ \l_{D,\gp}(f^nP) > 0\\ }} \log N_\gp
    - 
   \sum_{\substack{\gp\in\Spec R\\  \hw\l_{D,\gp}(f^mP)>0 \text{ for some $m<n$}\hw\\  }} 
          \log N_\gp \\
  &= N^{(1)}_S(D,f^nP) -
   \sum_{\substack{\gp\in\Spec R\\  \hw\l_{D,\gp}(f^mP)>0 \text{ for some $m<n$}\hw\\  }} 
          \log N_\gp \\
  &\ge N^{(1)}_S(D,f^nP) -
   \sum_{m=0}^{n-1} \sum_{\gp\in\Spec R} \l_{D,\gp}(f^mP) \\
  &\ge N^{(1)}_S(D,f^nP) - \sum_{m=0}^{n-1} h_D(f^mP) - O(1) \\*
  &\omit\hfill 
   since $D$ is effective, $\l_{D,v}$ is bounded below for all $v\in M_K$,\\
  &= N^{(1)}_S\bigl(((f^j)^*D)^\red,f^{n-j}(P)\bigr)
          - \sum_{m=0}^{n-1} h_D(f^mP) - O(1) 
    \quad\text{from \eqref{eqn:NDfjPeqNfjDredP},}\\
%%   &\omit\hfill from \eqref{eqn:NDfjPeqNfjDredP}, \\
  &= N^{(1)}_S\bigl(\D_j,f^{n-j}(P)\bigr) 
          - \sum_{m=0}^{n-1} h_D(f^mP) - O(1) \\*
  &\omit\hfill using our notation $\D_j=((f^j)^*D)^\red$, \\
  &\ge h_{\CanDiv_X+\D_j}\bigl(f^{n-j}(P)\bigr) - \e h_H\bigl(f^{n-j}(P)\bigr)
          - \sum_{m=0}^{n-1} h_D(f^mP) - O(1) \\*
  &\omit\hfill from Vojta's conjecture
        (Conjecture~\ref{conjecture:vojtatrunc}),\\
  &\ge \bigl(\deg(\CanDiv_X+\D_j)\bigr)-2\e)h_H\bigl(f^{n-j}(P)\bigr) \\*
    &\hspace{3em}{}      - \sum_{m=0}^{n-1} (\deg D + \e)h_H(f^mP) - O(1) 
        \quad\text{from Proposition \ref{proposition:htforalgeq},}\\
  &\ge \bigl(\deg(\CanDiv_X+\D_j)-2\e\bigr)\hhat_{f,H}\bigl(f^{n-j}(P)\bigr) \\*
    &\hspace{3em}{}     - \sum_{m=0}^{n-1} (\deg D + \e)\hhat_{f,H}(f^mP) - O(n) 
      \quad\text{from Theorem~\ref{theorem:canht}(a),}\\
  &= \bigl(\deg(\CanDiv_X+\D_j)-2\e\bigr)d^{n-j}\hhat_{f,H}(P) \\*
    &\hspace{3em}{} - \sum_{m=0}^{n-1} (\deg D + \e)d^m\hhat_{f,H}(P) - O(n) 
      \quad\text{from Theorem~\ref{theorem:canht}(b),}\\
  &\ge \left(\frac{\deg \CanDiv_X + \deg \D_j -2\e}{d^j}
        - \frac{\deg D+\e}{d-1}\right)d^n\hhat_{f,H}(P) - O(n).
\end{align*}
Using the assumed bound~\eqref{eqn:degDjdjgt} on~$\deg\D_j$, we see
that if we choose~$\e$ sufficiently small, then there is a
constant~$\kappa>0$ that does not depend on~$n$ such that
\[
  B_n \ge \kappa d^n\hhat_{f,H}(P) - O(n).
\]
Since we have also assumed that $\hhat_{f,H}(P)>0$, this proves that
the sum~$B_n$ defined by~\eqref{eqn:issumposq} is positive for all
sufficiently large~$n$ such that \text{$f^n(P)\notin Y_{f,D,j}$}
and~$f^n(P)$ is not one of the finitely many exceptions in Vojta's
conjecture. It follows that all such~$f^n(P)$ have a primitive divisor
relative to~$D$, and hence they are not in the Zsigmondy set
$\Zcal\bigl( \Orbit_f(P), D )$.  This completes the proof that
there are only finitely many $n\in\Zcal\bigl( \Orbit_f(P), D \bigr)$
such that $f^n(P)\notin  Y_{f,D,j}$.
\end{proof}

%%%%%%%%%%%%%%%%%%%%%%%%%%%%%%%%%%%%%%%%%%%%%%%%%%%%%%%%%%%%%%%%%%%%%%
\section{Primitive Divisors and Wandering Targets in Algebraic Groups}
\label{section:primdivalggps}
%%%%%%%%%%%%%%%%%%%%%%%%%%%%%%%%%%%%%%%%%%%%%%%%%%%%%%%%%%%%%%%%%%%%%%

Classical Zsigmondy problems are associated with algebraic groups.
Thus let~$G_K$ be a (commutative) algebraic group, let~$G_R$ be a model
for~$G_K$ over~$R$, and let~$\g\in G_K(K)$. In this setting, one can
study the Zsigmondy set associated to the sequence of
ideals~$\ga_n=\ga_n(\g)\subset R$ defined by the property that~$\ga_n$
is the smallest ideal with the property that
\[
  \g^n \equiv 1 \pmod{\ga_n}
  \quad\text{in $G_R(R)$.}
\]
The sequence~$(\ga_n)$ is a divisibility sequence, i.e., if~$m\mid n$,
then~$\ga_m\mid\ga_n$.  See~\cite[Section~6]{MR2162351}, and
especially~\cite[Proposition~8]{MR2162351}, for the general set-up.
If~$G_K=\GG_m$, then we obtain classical divisibility sequences such
as~$u^n-v^n$, while if~$G_K$ is an elliptic curve, then we obtain
classical elliptic divisibility sequences~\cite{MR0023275}. In both
cases, Zsigmondy sets are finite. 
\par
The situation for~$G_K=\GG_m^2$ is very different.  In this case we
obtain sequences such as $a_n=\gcd(u^n-1,v^n-1)$, for which Ailon and
Rudnick conjecture that there are infinitely many~$n$ such that
$a_n=1$, so in particular the Zsigmondy set is
infinite~\cite{MR2046966}.  In the opposite direction, Bugeaud,
Corvaja, and Zannier~\cite{MR1953049} use deep methods to prove that
$\lim_{n\to\infty}n^{-1}\log a_n=0$; and the author~\cite{MR2162351}
showed that Vojta's conjecture implies a far-reaching generalization
of the results in~\cite{MR1953049}.
\par
More generally, if $\dim(G_K)\ge2$, then excluding certain obvious
degenerate situations, one expects that the Zsigmondy set of the
sequence~$(\ga_n)$ should be infinite. Intuitively, this should be
true because $R$-valued points in~$G_R(R)$ are subschemes of
dimension~$1$, so two such points are unlikely to intersect in~$G_R$ if
$\dim(G_R)\ge3$. (Note that $\dim(G_R)=\dim(G_K)+1$, since~$\Spec(R)$
has dimension~$1$.)
\par
The underlying reason that the sequence~$(\ga_n)$ is a divisibility
sequence is because the target point, namely the identity element
of~$G$, is fixed by the~$n^{\text{th}}$-power maps that are being
applied to~$\g$.  If we instead choose a target point~$\b\in G$ that
does not have finite order and define~$\gb_n$ as the smallest ideal
such that
\[
  \g^n \equiv \b \pmod{\gb_n}
  \quad\text{in $G_R(R)$,}
\]
then~$(\gb_n)$ will not be a divisibility sequence, and even in the
case that~$\dim(G)=1$ it is a difficult question to determine if
the Zsigmondy set of~$(\gb_n)$ is finite.
\par
But in general, if we want to associate to the sequence of points~$\g^n$ a
Zsigmondy set that has a Zariski non-density property, then
we should pair the
point~$\g$, which has dimension~$1$ in~$G_R$, with a divisor~$D$,
which has codimension~$1$. This is what we did in
Section~\ref{section:notation}. We thus fix an ample effective
divisor~$D_K\in\Div(G_K)$, we let~$\G=(\g^n)_{n\ge0}$, and we consider
the Zsigmondy set~$\Zcal(\G,D)$. Taking~$G=\GG_m$ and $D_K=(1)$, or
taking~$G$ an elliptic curve and $D_K=(O)$, we recover the classical
divisibility sequences whose Zsigmondy sets are finite. For higher
dimensional~$G$, if the divisor~$D_K$ is not fixed by the powering map
in the group, it seems a difficult question to
characterize~$\Zcal(\G,D)$, even if one assumes Vojta's conjecture.
The difficulty comes from two facts.  First, the sequence of
ideals~$(\g^n,D)_R$ is not a divisibility sequence, so when checking
if~$\g^n$ has a primitive prime divisor, we must exclude all primes
dividing terms with~$m<n$. (If it were a divisibility sequence, we
would only need to consider the terms with~$m\mid n$.) Second, the
terms in the sequence~$\log\Norm_{K/\QQ}(\g^n,D)_R$ grow only polynomially
in~$n$, so the magnitude of the terms with~$m<n$ overwhelms
the~$n^{\text{th}}$ term. This is in contrast to the dynamical
setting, where the sequence~$\log\Norm_{K/\QQ}(f^n(P),D)_R$ grows
exponentially, which allows the~$n^{\text{th}}$ term to overwhelm the
terms with $m<n$. We thus raise the following question, which has an
affirmative answer for one-dimensional groups and $D_K=(1)$, but for
which there is currently no good evidence pro or con in higher
dimensions.

\begin{question}
Let $G_K$ be a commutative algebraic group, let~$D_K$ be an ample
effective normal crossings divisor, let~$\g\in G_K(K)$ be a wandering
point, and let $\G=(\g^n)_{n\ge0}$.  Is it true that the set
\[
  \bigl\{\g^n : n\notin\Zcal(\G,D)\bigr\}
\]
is not Zariski dense in~$G_K$?
\end{question}

%%%%%%%%%%%%%%%%%%%%%%%%%%%%%%%%%%%%%%%%%%%%%%%%%%%%%%%%%%%%%%%%%%%%%%
\section{Primitive Divisors for Cycles Having Complementary Codimension}
\label{section:primcycles}
%%%%%%%%%%%%%%%%%%%%%%%%%%%%%%%%%%%%%%%%%%%%%%%%%%%%%%%%%%%%%%%%%%%%%%
As noted in Section~\ref{section:primdivalggps}, one reason that we
expect Zsigmondy sets $\Zcal(\Qcal,D)$ to be small is because the
$R$-points~$Q_n$ have dimension~$1$ in~$X_R$ and the divisor~$D_R$ has
codimension~$1$ in~$X_R$, so as the size of~$Q_n$ grows, the
intersection ideal of~$Q_n$ and~$D_R$ should grow. This suggests
taking other sequences of subschemes of~$X_R$ having complementary
codimensions.

\begin{definition}
Let $K$,~$X_K$,~$R$, and~$X_R$ be as defined in
Section~\ref{section:notation}. Let~$\xi_K$ and~$\eta_K$ be
equidimensional subvarieties of~$X_K$ such that
\begin{equation}
  \label{eqn:codimdim1}
  \xi_K\cap\eta_K=\emptyset
  \quad\text{and}\quad
  \codim(\xi_K) + \codim(\eta_K) = \dim(X_K) + 1.
\end{equation}
Let~$\xi_R$ and~$\eta_R$ be the closures in~$X_R$ of~$\xi_K$
and~$\eta_K$, respectively. For $\gp\in\Spec(R)$, we set
\[
  i_\gp(\xi,\eta) = 1
  \quad\text{if}\quad
  (\tilde\xi_R\bmod\gp) \cap (\tilde\eta_R\bmod\gp) \ne \emptyset,
\]
and otherwise we set $i_\gp(\xi,\eta) = 0$. In other words, 
$i_\gp(\xi,\eta)$ is the indicator function for whether~$\xi$ and~$\eta$
intersect when they are reduced modulo~$\gp$. 
\end{definition}

\begin{definition}
Let $\Xcal=(\xi_n)_{n\ge0}$ and $\Ncal=(\eta_n)_{n\ge0}$ be sequences
of equidimensional subvarieties of~$X_K$
satisfying~\eqref{eqn:codimdim1}. We say that~$\gp\in\Spec(R)$ is a
\emph{primitive} (\emph{prime}) \emph{divisor of the pair~$(\xi_n,\eta_n)$} if
\[
  i_\gp(\xi_n,\eta_n)=1
  \qquad\text{and}\qquad
  i_\gp(\xi_m,\eta_m)=0
  \quad\text{for all $0\le m<n$.}
\]
Then the \emph{Zsigmondy set of the sequences~$\Xcal$ and~$\Ncal$} is
\[
  \Zcal(\Xcal,\Ncal) = \{n\ge0 : 
    \text{the pair $(\xi_n,\eta_n)$ has no primitive divisors}\}.
\]
\end{definition}

We make a specific conjecture in the dynamical setting.  This
conjecture can certainly be generalized, but it is not entirely clear
how to characterize the ``degenerate'' cases, so we are content with
the following modest conjecture.

\begin{conjecture}
\label{conjecture:codimge1}
Let $f:\PP^N_K\to\PP^N_K$ be a morphism of degree $d\ge2$.  Let~$\xi$
and~$\eta$ be smooth irreducible subvarieties of~$\PP^N_K$ 
with the following properties\textup:
\begin{parts}
\Part{\textbullet}
$\codim(\xi)+\codim(\eta)=N+1$.  
\Part{\textbullet}
$f^n(\xi)\cap\eta=\emptyset$ for all $n\ge0$.
\Part{\textbullet}
For any infinite sequence of integers~$n_i\ge0$, the
union of the varieties~$f^{n_i}(\xi)$ is Zariski dense in~$\PP^N$.
\end{parts}
Let~$\Xcal=\bigl(f^n(\xi)\bigr)$ be the sequence consisting of the
$f$-orbit of~$\xi$, and let~$\Ncal=(\eta)$ be the constant
sequence. Then the Zsigmondy set $\Zcal(\Xcal,\Ncal)$ is finite.
\end{conjecture}

\begin{remark}
In the setting of Conjecture~\ref{conjecture:codimge1}, one might
consider a second morphism~$g:\PP^N_K\to\PP^N_K$ and replace the
constant sequence~$\Ncal=(\eta)$ with the moving
sequence~$\Ncal=\bigl(g^n(\eta)\bigr)$.  Then the question of whether
$\Zcal(\Xcal,\Ncal)$ is finite will depend in some complicated way on
how~$f$ and~$g$ interact dynamically. This should lead to all sorts of
interesting questions.
\end{remark}

%%%%%%%%%%%%%%%%%%%%%%%%%%%%%%%%%%%%%%%%%%%%%%%%%%%%%%%%%%%%%%%%%%%%%%%%
% Bibliography
%%%%%%%%%%%%%%%%%%%%%%%%%%%%%%%%%%%%%%%%%%%%%%%%%%%%%%%%%%%%%%%%%%%%%%%%

%% \begin{thebibliography}{99}
%% \itemsep=\smallskipamount
%% \end{thebibliography}

%% \bibliographystyle{plain}
%% \bibliography{c:/Users/Joe/JHS/Book/ADS/ArithDyn}

\end{document}